\documentclass[12pt]{article}
\usepackage{amsmath,latexsym}
\usepackage{amsthm}
\usepackage{amssymb}
\usepackage{vatola}

\def\q{\quad}
\def\qq{\qquad}
\def\t{\hbox}
\def\e{\equiv}
\def\f{\frac}
\def\b{\binom}
\def\a{\alpha}
\def\ord{\t{ord}}
\def\mod#1{\ (\hbox{\rm mod}\ #1)}

\TagsOnRight
\def\qtq#1{\q\t{#1}\q}
\theoremstyle{plain}
\newtheorem{thm}{Theorem}[section]
\theoremstyle{definition}

\theoremstyle{lemma}
\newtheorem{lem}{Lemma}[section]

\newtheorem{cor}{Corollary}[section]

\begin{document}

\title{Extensions of Stern's congruence for Euler numbers}
\author{
   Zhi-Hong Sun$^1$, Long Li$^2$\\
   $^1$ School of Mathematical Sciences,  Huaiyin Normal University,
 \\Huaian 223001, People's Republic of China \\
 E-mail: zhihongsun@yahoo.com\\
 $^2$ School of Mathematics and Statistics, Jiangsu Normal University,
 \\Xuzhou 221116, People's Republic of China \\
  E-mail:  lilong6820@126.com }
\date{}          
\maketitle
 \abstract{For a nonzero integer $a$ let ${E_n^{(a)}}$  be given by
 $\sum_{k=0}^{[n/2]}\binom
n{2k}a^{2k}E_{n-2k}^{(a)}=(1-a)^n$ $(n=0,1,2,...)$, where $[x]$ is
the greatest integer not exceeding $x$. As $E_n^{(1)}=E_n$ is the
Euler number, $E_n^{(a)}$ can be viewed as a generalization of Euler
numbers. Let $k$ and $m$ be positive integers, and let $b$ be a
nonnegative integer. In this paper, we determine $E_{2^mk+b}^{(a)}$
modulo $ 2^{m+10}$ for $m\ge 5$. For $m\ge 5$ we also establish
congruences for $U_{k\varphi{(5^m)}+b},\; E_{k\varphi{(5^m)}+b},\;
S_{k\varphi{(5^m)}+b}\mod{5^{m+5}}$ and
$S_{k\varphi{(3^m)}+b}\mod{3^{m+5}},$ where $U_{2n}=E_{2n}^{(3/2)}$,
$S_n=E_n^{(2)}$ and $\varphi(n)$ is Euler's function.
\par\q
\newline MSC: 11B68, 11A07
\newline Keywords: Congruence, Euler number
}
 \endabstract
\section{Introduction}
\par
Let $\Bbb Z$ and $\Bbb N$ be the set of integers and the set of
positive integers, respectively. The Euler numbers $\{E_n\}$ are
given by
$$E_0=1,\q
E_{2n-1}=0,\q\sum_{r=0}^n\b{2n}{2r}E_{2r}=0\q(n=1,2,3,\ldots).$$ For
$k,m\in\Bbb N$ and $b\in\{0,2,4,\ldots\}$, in 1875 Stern [9] proved
the following congruence, which is now known as Stern's congruence:
$$E_{2^mk+b}\e E_b+2^mk\mod{2^{m+1}}.\tag 1.1$$
There are many modern proofs of (1.1). See for example [1,3,8,10].
\par
Let $b\in\{0,2,4,\ldots\}$ and $k,m\in\Bbb N$. In [7] the first
author and L.L. Wang showed that
$$E_{2^mk+b}\e E_b+
2^mk(7(b+1)^2-18 +2^mk(7-b))\mod {2^{m+7}}\qtq{for}m\ge 3.\tag 1.2$$
In [4], the first author introduced the sequence $\{E_n^{(a)}\}$ as
a generalization of Euler numbers. For $a\not=0$, $\{E_n^{(a)}\}$ is
given by
$$\sum_{k=0}^{[n/2]}\binom n{2k}a^{2k}E_{n-2k}^{(a)}=(1-a)^n\q
(n=0,1,2,...).\tag 1.3$$ Clearly $E_n=E_n^{(1)}$. The first few
values of $E_n^{(a)}$ are shown below:
 $$\aligned &E_0^{(a)}=1,\q E_1^{(a)}=1-a,\q
E_2^{(a)}=1-2a,\q E_3^{(a)}=1-3a+2a^3,\\&E_4^{(a)}=1-4a+8a^3,\q
E_5^{(a)}=1-5a+20a^3-16a^5,\\& E_6^{(a)}=1-6a+40a^3-96a^5,\\&
E_7^{(a)}=1-7a+70a^3-336a^5+272a^7,\\&
E_8^{(a)}=1-8a+112a^3-896a^5+2176a^7.\endaligned\tag 1.4$$ For a
prime $p$ and nonzero integer $n$ let $\ord_pn$ denote the unique
nonnegative integer $\alpha$ such that $p^{\a}\mid n$ and
$p^{\a+1}\nmid n$.
 In [4],
Z.H. Sun showed that for given nonzero integer $a$, nonnegative
integer $b$ and positive integers $k$ and $m$,
$$\aligned &E_{2^mk+b}^{(a)}- E_b^{(a)}
 \\&\e \cases 2^mk(a^3((b-1)^2+5)-a+2^mka^3(b-1))
 \mod{2^{m+4+3\ord_2a}}& \t{if $2\mid a$,}
\\ 2^mka((b+1)^2+4-2^mk(b+1))
 \mod{2^{m+4}}&\t{if $2\nmid a$ and $2\mid b$,}
\\ 2^mk(a^2-1)
 \mod{2^{m+4}}&\t{if $2\nmid ab$.}
 \endcases\endaligned\tag 1.5$$
In this paper we determine $E_{2^mk+b}^{(a)}-
E_b^{(a)}\mod{2^{m+10}}$ for $m\ge 4$. See Theorems 2.1-2.3.
\par For a real number $x$ let $[x]$ be the greatest integer
 not exceeding $x$. Let $\{U_n\}$ and $\{S_n\}$ be given by
 $$\align &U_0=1,\q U_n=-2\sum_{k=1}^{[n/2]}\binom n{2k}U_{n-2k}\q
  (n\geq1),
  \\&S_0=1,\q S_n=1-\sum_{k=0}^{n-1}\b nk2^{2n-2k-1}S_k\q(n\ge 1).
  \endalign$$
 From [4,5] we know that
$$U_{2n}=E_{2n}^{(3/2)}\qtq{and}S_n=E_n^{(2)}.\tag 1.6$$
Let $\varphi(n)$ denote Euler's totient function. In [6] the first
author established congruences for $E_{k\varphi(3^m)+b}
\mod{3^{m+4}}$ and $U_{k\varphi(3^m)+b} \mod{3^{m+4}}.$
\par For $m\in \Bbb N$, let $\Bbb Z_m$ be the set of rational numbers
whose denominator is
  coprime to $m$. In [2], the first author introduced the notion of $p-$regular
   functions. If $f(k)\in \Bbb
Z_p$ for any nonnegative integers $k$ and $\sum_{k=0}^n\binom
nk(-1)^kf(k)\equiv 0\mod{p^n}$ for all $n \in \Bbb N$, then $f$ is
called a $p-$regular function.
\par
Let $(\frac a p)$ be the Legendre symbol. In [5] and [3], the first
author proved that for any odd prime $p$, both $f_1(k)=(1-(\frac p
3)p^{k(p-1)+b})U_{k(p-1)+b}$ and
$f_2(k)=(1-(-1)^{\frac{p-1}2}p^{k(p-1)+b})E_{k(p-1)+b}$ are
$p-$regular functions. In Sections 3-6, using $p$-regular function
and the binomial inversion formula, we
  obtain congruences for $$U_{k\varphi(5^m)+b}
,\ E_{k\varphi(5^m)+b},\ S_{k\varphi(5^m)+b}
\mod{5^{m+5}}\qtq{and}S_{k\varphi(3^m)+b}\mod{3^{m+5}},$$ where
$k,m\in\Bbb N$, $m\ge 5$ and $b$ is a nonnegative integer. See
Theorems 3.1, 4.1, 5.1 and 6.1.

\section{Congruences for $E_{2^mk+b}^{(a)}\mod{2^{m+10}}$}
\par\q  For $a \not =0$ let
$\{E_n^{(a)}\}$ be given by (1.3), and let $$
e_s(a,b)=2^{-s}\sum_{r=0}^s\binom sr(-1)^rE_{2r+b-2[\frac b
2]}^{(a)}\tag 2.1$$ for $s\in\Bbb N$. Then the first few $e_s(a,b)$
are as follows:
$$\aligned
&e_1(a,b)=\cases a\q&\text{if $2\mid b$},\\a-a^3 &\text{if $2\nmid
b$}.\endcases\\& e_2(a,b)=\cases 2a^3\q&\text{if $2\mid b,$}
 \\4a^3(1-a^2)&\text{if $2\nmid b$},\endcases
 \\& e_3(a,b)=\cases 2a^3(6a^2-1)\q &\text{if $2\mid b$},
 \\2a^3(-17a^4+18a^2-1)&\text{if $2\nmid b$},\endcases
 \\& e_4(a,b)=\cases 8a^5(17a^2-4)\q&\text{if $2\mid b$},
 \\-48a^5+544a^7-496a^9&\text{if $ 2\nmid b$},\endcases
\\& e_5(a,b)=\cases 16a^5-680a^7+2480a^9\q&\text{if $2\mid b$},
\\16a^5-1360a^7+12400a^9-11056a^{11}&\text{if $2\nmid b$},
\endcases
\\& e_6(a,b)=\cases 816a^7-19840a^9+66336a^{11}\q&\text{if $2\mid b$},
\\1088a^7-49600a^9+398016a^{11}-349504a^{13}&\text{if $2\nmid b$}
\endcases\\&e_7(a,b)=\cases -272a^7+41664a^9-773920a^{11}+2446528
a^{13}\q&\text{if $2\mid b$},\\
-272a^7+69440a^9-2321760a^{11}+17125696a^{13}-14873104a^{15}
&\text{if $2\nmid b$}.\endcases\endaligned\tag 2.2$$
\newtheorem{theorem}{Theorem}
\newtheorem{corollary}{Corollary}
\begin{lem}$($See \rm{[4, Theorem 2.1]}) Let $n$ be a nonnegative integer
and $a\not=0$. Then
$$E_n^{(a)}=\sum_{k=0}^{[n/2]}\binom n{2k}(1-a)^{n-2k}a^{2k}E_{2k}.$$
\end{lem}
\begin{lem}$($See \rm{[4, Theorem 3.1]}) Let a be a nonzero integer,  $n\in \Bbb N$ and let b
be a nonnegative integer. Suppose that $\alpha_n$ is a nonnegative
integer given by $2^{\alpha_n-1}\leq n <2^{\alpha_n}$.
 \par $\t{\rm
(i)}$ We have
 $$\aligned \sum_{k=0}^n \binom n k (-1)^{n-k}E_{2k}^{(a)}\equiv \cases 0
  \mod {2^{(n+1)\text{\rm{ord}}_2a-\alpha_n+\text{\rm{ord}}_2n+2n}}
  &\text{if $2\mid n$},\\
0\mod{2^{n\text{\rm{ord}}_2a+2n-\alpha_n}}&\text{if $2\nmid n$}.
\endcases\endaligned$$
and
$$
\aligned\sum_{k=0}^n\binom n k (-1)^{n-k}E_{2k+1}^{(a)}\equiv \cases
0\mod{2^{(n+1)\text{\rm{ord}}_2a+2n}}&\text{if $2\mid n$},
\\
0\mod{2^{n\text{\rm{ord}}_2a+2n-\text{\rm{ord}}_2{(n+1)}}}&\text{if
$2\nmid n$}.
\endcases\endaligned$$
\par $\t{\rm(ii)}$ We have
$$\sum_{k=0}^n\binom n k(-1)^kE_{2k+b}^{(a)}\equiv 0 \mod{2^{(2+\text{\rm{ord}}_2a)n-\alpha_n}}.
$$
\end{lem}
\begin{lem}$($See \rm{[4, (3.4)]}$)$ Let $n\in \Bbb N$, $a\in\Bbb Z$,
$a\not=0$, $b\in\{0,1,2,...\}$.
 Then
$$\sum_{r=0}^n\binom nr(-1)^rE_{2r+b}^{(a)}=\sum_{r=0}^{[\frac b 2]}\binom{[\frac b
2]}r(-1)^r\sum_{s=0}^{r+n}\binom{r+n}s(-1)^sE_{2s+b-2[\frac b
2]}^{(a)}.$$
\end{lem}
\begin{lem} Let $a$ be a nonzero integer, $k,m\in \Bbb N$ and
 $ b \in \{0,1,2,...\}$. Then
$$E_{2^mk+b}^{(a)}-E_b^{(a)}\equiv
2^{m-1}k\sum_{r=1}^8\binom{2^{m-1}k-1}{r-1}\frac{(-2)^r}r A_r(a,b)
\mod {2^{m+13+9\text{\rm{ord}}_2a}},$$ where
$$A_r(a,b)=2^{-r}\sum_{i=0}^{r}\binom ri(-1)^i E_{2i+b}^{(a)}.\tag 2.3$$
\end{lem}
\begin{proof} Let $e_s(a,b)$ and $A_s(a,b)$
   be given by (2.1) and (2.3). Then $E_b^{(a)}=A_0(a,b)$.
 From the binomial inversion formula, we get
$$\sum_{r=0}^s\binom sr\cdot(-1)^r\cdot A_r(a,b)\cdot 2^r=E_{2s+b}^{(a)}.$$
Hence $$\aligned E_{2^mk+b}^{(a)}-E_b^{(a)}
&=\sum_{r=0}^{2^{m-1}k}\binom{2^{m-1}k}r\cdot(-1)^r\cdot 2^r\cdot A_r(a,b)-A_0(a,b)\\
&=\sum_{r=1}^{2^{m-1}k}\binom{2^{m-1}k}r\cdot(-1)^r\cdot 2^r\cdot A_r(a,b)\\
&=2^{m-1}k\sum_{r=1}^{2^{m-1}k}\binom{2^{m-1}k-1}{r-1}\cdot(-1)^r\cdot\frac{2^r}r\cdot
A_r(a,b).\endaligned$$\par
 From  Lemma 2.2(ii) we have
 $$ 2^{(2+\text{ord}_2a)r-\alpha_r}\bigm| \sum_{i=0}^r\binom r i(-1)^iE_{2i+b}^{(a)}.$$
Thus $2^{(1+\text{ord}_2a)r-\alpha_r}\mid A_r(a,b).$\par
 For $r
\ge9$ we have
$\text{ord}_2A_r(a,b)\ge9(1+\text{ord}_2a)-\alpha_9=9(1+\text{ord}_2a)-4=5+9\text{ord}_2a,$\\
 and
$\frac{2^r}r\equiv 0 \mod {2^9}.$
 So
$$\align  E_{2^mk+b}^{(a)}-E_b^{(a)}&=2^{m-1}k\sum_{r=1}^{2^{m-1}k}\binom{2^{m-1}k-1}{r-1}\cdot(-1)^r\cdot\frac{2^r}r\cdot A_r(a,b)\\
&\equiv
2^{m-1}k\sum_{r=1}^8\binom{2^{m-1}k-1}{r-1}\cdot(-1)^r\cdot\frac{2^r}r\cdot
A_r(a,b) \mod {2^{m+13+9\text{ord}_2a}}.\endalign$$ This yields the
result.\end{proof}
\begin{thm} Let $a$ be a nonzero  integer, $k,m\in\Bbb N$,
$m\geqslant 4$ and $b \in \{0,1,2,\ldots\}$. If $2\mid a$, then
$$\aligned E_{2^mk+b}^{(a)}-E_b^{(a)}
\equiv\cases 2^{m}k\{9a^3b^2-(2a^3-128a)b+86a^3-257a
\\ \q+2^{m}ka^3(b-1)\}\mod{2^{m+10}}\qq\qq\qq\qq\q\qq\q\ \; \text{\rm{if} $2\mid b$},\\
2^{m}k\{(a^3-2a^5)b^2+(30a^3+128a)b-2a^5+86a^3+127a\\
 \q+2^{m}ka^3(b-1)\mod{2^{m+10}}
 \q\qq\qq\qq\qq\qq\qq\text{\rm{if} $2\nmid b$}.
\endcases\endaligned$$

\end{thm}
\begin{proof} From Lemma 2.2(i)  we have
$$\sum_{s=0}^{r+8}\binom {r+8} s (-1)^s E_{2s+b-2[\frac b 2]}^{(a)}
\equiv 0\mod {2^{14+9\text{ord}_2a}}\qtq{for} r\in \{0,1,2,...\}.$$
Thus,
 $$A_8(a,b)=\frac 1 {2^8}\sum_{r=0}^{[b/2]}\binom {[\frac b 2]}r(-1)^r\sum_{s=0}^{r+8}\binom {r+8} s (-1)^s
 E_{2s+b-2[\frac b 2]}^{(a)}\equiv 0 \mod {2^{6+9\text{ord}_2a}}.$$
From the above and Lemma 2.3 we see that for $n\le 7$,
 $$\align A_n(a,b)&=\frac 1{2^n}\sum_{r=0}^{[\frac b 2]}\binom{[\frac b 2]}r(-1)^r\sum_{s=0}^{r+n}\binom{r+n}s(-1)^sE_{2s+b-2[\frac b 2]}^{(a)}\\
&\equiv \frac 1 {2^n}\sum_{r=0}^{7-n}\binom {[\frac b
2]}r(-1)^r\sum_{s=0}^{r+n}\binom {r+n}s(-1)^sE_{2s+b-2[\frac b
2]}^{(a)} \mod{2^{14-n+9\text{ord}_2a}}. \endalign$$ That is,
$$A_n(a,b)\e \sum_{r=0}^{7-n}\binom {[\frac b
2]}r(-2)^re_{r+n}(a,b)\mod{2^{14-n+9\text{ord}_2a}}.\tag 2.4$$
 Hence, using (2.2) we see that
$$\aligned &A_7(a,b)\e e_7(a,b)\equiv 0 \mod {2^4},\\
&A_6(a,b)\e e_6(a,b)-2[\frac b 2]e_7(a,b)\equiv\cases 0 &\mod
{2^6}\quad \text{if $2\mid a$},\\16(3a+b+2)&\mod {2^6}\quad \text{if
$2\nmid a$ and $2\mid b$},\\0&\mod {2^6}\quad \text{if
$2\nmid{ab}$}.\endcases
\endaligned
$$
Similarly,
$$
\aligned A_5(a,b)&\e e_5(a,b)-2\binom{[\frac b 2]}1e_6(a,b)+4\binom{[\frac b 2]}2e_7(a,b)\\
&\equiv\cases 0&\mod {2^6}\q \text{if $2\mid a$},\\24a-8b^2 &\mod
{2^6}\q \text{if $2\nmid a $ and $ 2\mid b$},\\0&\mod {2^6}\q
\text{if $2\nmid{ab}$},\endcases \\ A_4(a,b)&\e e_4(a,b)-2[\frac b
2]e_5(a,b)+4\binom{[\frac b 2]}2e_6(a,b)-8\binom{[\frac b
2]}3e_7(a,b)\\&\equiv 0 \mod{2^9}\qtq{for even $a$,}\\ A_3(a,b)&\e e_3(a,b)-2[\frac b 2]e_4(a,b)+4\binom{[\frac b 2]}2e_5(a,b)-8\binom{[\frac b 2]}3e_6(a,b)+16\binom{[\frac b 2]}4e_7(a,b)\\
&\equiv  -2a^3+4a^5\mod{2^8} \q\t{for even $a$ }\\
\endaligned
$$ and
$$\aligned A_2(a,b)&\e e_2(a,b)-2[\frac b 2]e_3(a,b)+4\binom{[\frac b 2]}2e_4(a,b)-8\binom{[\frac b 2]}
3e_5(a,b)+16\binom{[\frac b 2]}4e_6(a,b)\\&\q-32\binom{[\frac b 2]}5e_7(a,b)\\
&\equiv \cases 2a^3-\frac 4 3 a^5b+2a^3b\mod{2^{10}}\q&\text{if $2\mid a$
and $2\mid b$ },\\2a^3-4a^5b+2a^3b-6a^7+6a^7b\mod{2^{10}}\qquad&\text{if $2\mid a $
and $ 2\nmid b$}.\\
\endcases
\endaligned$$
\par When $2\mid a$ and $ 2\mid b$, we have
$$\align &-2A_1(a,b)\equiv \frac 8 3 a^5b-\frac 4 3 a^5b^2+2a^3b^2-2a\equiv
24a^5b-12a^5b^2+2a^3b^2-2a \mod {2^{11}},\\&\frac {2^2}2\cdot
A_2(a,b)=2A_2(a,b)\equiv(4a^3-24a^5b+4a^3b)\mod{2^{11}},\\&-\frac
{2^3}3\cdot A_3(a,b)\equiv -\frac 8 3 (-2a^3+4a^5)\equiv
176a^3-32a^5\mod{ 2^{11}}.\endalign$$ Now combining Lemma 2.4 with
the above we obtain
$$ \align &E_{2^mk+b}^{(a)}-E_b^{(a)}\\&\equiv 2^{m-1}k\sum_{r=1}^8\binom{2^{m-1}k-1}{r-1}\cdot (-1)^r\cdot\frac {2^r} r\cdot A_r(a,b).\\
&\equiv 2^{m-1}k\{(24a^5b-12a^5b^2+2a^3b^2-2a)+(2^{m-1}k-1)(4a^3-24a^5b+4a^3b)\\&\q+\binom{2^{m-1}k-1}2(176a^3-32a^5)\}\\
&\equiv
2^{m-1}k\{48a^5b-12a^5b^2+2a^3b^2-2a-4a^3-4a^3b+176a^3-32a^5\\&\q+2^{m+1}k(b-1)a^3\}\mod{
2^{m+10}}.\endalign$$ To see the result, we note that
$$\align &32a^5=2^{10}(\f a2)^5\e 2^{10}\f a2=512a\mod{2^{11}},\
48a^5b=3\cdot2^{10}(\f a2)^5\f b2\e 256ab\mod{2^{11}},\\
&-12a^5b^2=-3\cdot2^9(\f a2)^5(\f b2)^2\e
16a^3b^2\mod{2^{11}}.\endalign$$
\par When $2\mid a $ and $2\nmid b$, we have
$$\align &-2A_1(a,b)\equiv 8a^5b-4a^5-4a^5b^2+2a^3b^2-2a\mod {2^{11}},\\
&\frac {2^2}2A_2(a,b)=2A_2(a,b)\equiv 4a^3(b+1)-8a^5b+12a^7(b-1)\mod {2^{11}},\\
&-\frac {2^3}3A_3(a,b)\equiv-\frac 8 3(-2a^3+4a^5)\equiv
176a^3-32a^5\mod{2^{11}}.\endalign$$ Combining Lemma 2.4 with the
above we get
$$\align &E_{2^mk+b}^{(a)}-E_b^{(a)}\\&\equiv 2^{m-1}k\sum_{r=1}^8\binom{2^{m-1}k-1}{r-1}\cdot (-1)^r\cdot\frac {2^r} r\cdot A_r(a,b)\\
& \equiv 2^{m-1}k\{8a^5b-4a^5-4a^5b^2+2a^3b^2-2a+(2^{m-1}k-1)(4a^3(b+1)-8a^5b\\
&\q+12a^7(b-1))+\frac{(2^{m-1}k-1)(2^{m-1}k-2)}2(176a^3-32a^5)\}\\
&\equiv
2^{m-1}k\{16a^5b-4a^5-4a^5b^2+2a^3b^2-2a-12a^7(b-1)+176a^3-32a^5\\&\q+(2^{m+1}k-4)a^3(b+1)+2^{m+2}k(-33a^3)\}\mod{2^{m+10}}.\endalign$$
To see the result, we note that
$$\align &16a^5b=2^9(\f a2)^5b\e 2^9(\f a2)^3b=64a^3b\mod{2^{11}},\\
&-12a^7(b-1)=-3\cdot2^{10}(\f a2)^7\f {b-1}2\e
256ab-256a\mod{2^{11}}.\endalign$$
 This completes the proof.
\end{proof}

\begin{thm}Let $a$ be an odd integer, $k,m\in \Bbb N$, $m\geqslant 5$
and $ b \in \{0,2,4,\ldots\}$. Then
$$\align
E_{2^mk+b}^{(a)}-E_b^{(a)}&\equiv
2^{m}k\{7ab^6-6ab^5+(3a^3-14a)b^4+(4a^3+56a)b^3\\&\q-(6a^4+35a^3-12a^2+106a-122)b^2
\\&\q+(38a^3-8a-256)b+(70a^3+64a^2-81a+448)\\
&\q+2^{m}k(b^4+2b^3+2ab^2+(a^3+2a)b+16-a^3)\}\mod{2^{m+10}}.
\endalign$$

\end{thm}
\begin{proof} By (2.4) we have
$$\align A_1(a,b)&\e \sum_{r=0}^6\binom{[\frac b 2]}r(-2)^re_{r+1}(a,b)\\
&\equiv a-2a^3b+a^3b(b-2)(6a^2-1)-\frac 4 3
b(b-2)(b-4)(19a^3-6a)\\&\q+b(b-2)(b-4)(b-6)(15a^5+a^3+17a)-b(b-2)(b-4)(b-6)(b-8)(28a-16)\\&
\q -b(b-2)(b-4)(b-6)(b-8)(b-10)(-5a^3+14a)\mod{2^{10}},
\\A_2(a,b)&\e\sum_{r=0}^5\binom{[\frac b 2]}r(-2)^re_{r+2}(a,b)
\\&\equiv
2a^3-2a^3b(6a^2-1)+4a^5b(b-2)(17a^2-4)-b(b-2)(b-4)(-40a-84a^3)\\&\q
+b(b-2)(b-4)(b-6)(50a^3-20a)\\&\q-b(b-2)(b-4)(b-6)(b-8)(-30a^3-12a)\mod{2^{10}}.
\\ A_3(a,b)&\e\sum_{r=0}^4\binom{[\frac
b 2]}r(-2)^re_{r+3}(a,b)\\
&\equiv
2a^3(6a^2-1)-8ab(17a^2-4)+b(b-2)(32a-20a^3)-b(b-2)(b-4)(8a+16)\\&\q+b(b-2)(b-4)(b-6)(-2a)
\mod{2^8}\endalign$$ and $$\align
A_4(a,b)&\e\sum_{r=0}^3\binom{[\frac b
2]}r(-2)^re_{r+4}(a,b)\\&\equiv
8(13a^5+17a^3-17a)-b(16a+8a^3)+b(b-2)(24a^3+16a-64)\\&\q+8ab(b-2)(b-4)
\mod{2^9}.\endalign$$ By the proof of Theorem 2.1,
 $$\align& A_5(a,b)\equiv
24a-8b^2 \mod{2^6},\\&A_6(a,b)\equiv 16(3a+b+2)\mod{2^6},\q
A_7(a,b)\equiv 0\mod{2^4}.\endalign$$
 Since $m\geqslant 5$ we have
$$\align &2(2^{m-1}k-1)A_2(a,b)\\&\equiv
-2\big(2a^3-2a^3b(6a^2-1)+4a^5b(b-2)(17a^2-4)-b(b-2)(b-4)(-40a-84a^3)\\&\q
+b(b-2)(b-4)(b-6)(50a^3-20a)-b(b-2)(b-4)(b-6)(b-8)(-30a^3-12a)\big)
\\&\q+2^mk(2a^3-2b(6a-a^3)+52ab(b-2)
-4b(b-2)(b-4)\\&\qq+2b(b-2)(b-4)(b-6))\mod{2^{11}},\\
 &\binom{2^{m-1}k-1}2\cdot (-1)^3\cdot\frac {2^3}3\cdot A_3(a,b)\\
&=\frac{-4}3(2^{m-1}k-1)(2^{m-1}k-2)A_3(a,b)\\
&\equiv -\frac 8
3\big(2a^3(6a^2-1)-8ab(17a^2-4)+b(b-2)(32a-20a^3)\\&\q-b(b-2)(b-4)
(8a+16)+b(b-2)(b-4)(b-6)(-2a)\big)\\&\qq+2^{m+1}k(12a-2a^3+4b^2)\mod{2^{11}},\\
 &\binom{2^{m-1}k-1}3\cdot(-1)^4\cdot\frac {2^4} 4\cdot A_4(a,b)\\
&=\frac 2 3(2^{m-1}k-1)(2^{m-1}k-2)(2^{m-1}k-3)A_4(a,b)\equiv\frac {11}32^mkA_4(a,b)-4A_4(a,b)\\
&\equiv -4\big(104a^5+136a^3-136a-b(16a+8a^3)+b(b-2)(24a^3+16a-64)+8ab(b-2)(b-4)\big)\\&\q+\frac{11}32^mk(40a+8b^2-8ab)\mod{2^{11}},\\
&\binom{2^{m-1}k-1}4\cdot(-1)^5\cdot\frac {2^5}5\cdot A_5(a,b)\\
&=-\frac 4{15}(2^{m-1}k-1)(2^{m-1}k-2)(2^{m-1}k-3)(2^{m-1}k-4)A_5(a,b)\\
&\equiv \frac 5 3 2^{m+2}kA_5(a,b)-\frac{32}5A_5(a,b)\equiv \frac 5
32^{m+5}k -\frac{32}5(24a-8b^2)\mod{2^{11}}\endalign$$ and $$\align
&\binom{2^{m-1}k-1}5\cdot(-1)^6\cdot\frac {2^6}6\cdot A_6(a,b)\\
&=\frac{4}{45}\binom{2^{m-1}k-1}5A_6(a,b)\equiv \frac {137} {45}
2^{m+2}kA_6(a,b)-\frac{32}3A_6(a,b)\\&\equiv
-\frac{32}3(48a+16b+32)\equiv -512(a+b+2)\mod{2^{11}}.\endalign$$
Now combining the above congruences with Lemma 2.4 we deduce
that
$$\align
&E_{2^mk+b}^{(a)}-E_b^{(a)}\\&\equiv
2^{m-1}k\{-1024+286a+400a^5b-516a^3b-96ab-248ab^2+378a^3b^2+24ab^3\\&\q-8a^3b^3+6ab^4+40ab^5+10a^3b^6
+4ab^6-136a^7b^2-276a^5b^2+272a^7b\\&\q+140a^3-448a^5+256b^2-32b^5-30a^5b^4+12a^3b^5+2a^3b^4+104a^5b^3\\
&\q+2^{m+1}k(16+2ab-a^3+2b^3+b^4+a^3b+2ab^2)\}\mod{2^{m+10}}.
\endalign$$
By simplifying the above congruence we obtain the result.

\end{proof}
 \begin{cor} Let $k,m\in\Bbb N$, $m\ge 5$
and $b\in\{0,2,4,\ldots\}$. Then
$$\align E_{2^mk+b}-E_b&\equiv 2^mk\{7b^6-6b^5-11b^4+60b^3-13b^2-226b+501\\&\q
+2^mk(b^4+2b^3+2b^2+3b+15)\}\mod{2^{m+10}}\endalign$$\end{cor}
\begin{proof} Putting $a=1$ in Theorem 2.2  we deduce the
result.
\end{proof}
\begin{thm}Let $a$ be an  odd integer, $k,m\in \Bbb N$, $m\geqslant 5$ and
 $ b \in \{1,3,5,...\}$. Then
$$\align &E_{2^mk+b}^{(a)}-E_b^{(a)}\\&\equiv
2^mk\{(17a^7+162a^5+153a^3+64a^2+180a+192)b^2\\&\q
-(102a^7+216a^5+386a^3-704a)b-(211a^7-10a^5-32a^4+66a^3-267a-224)\\&\q+2^{m+1}(a^2-1)\}
\mod{2^{m+10}}.\endalign$$

\end{thm}
\begin{proof} By (2.4) we have
 $$\align &A_7(a,b)\equiv 0 \mod{2^4},\q A_6(a,b)\equiv 0\mod{2^6},\q A_5(a,b)\equiv 0\mod{2^6},
 \\&A_4(a,b)\equiv 0\mod{2^9},\q A_3(a,b)\equiv
-34a^7+36a^5-2a^3\mod{2^8}.\endalign$$
 Also, $$\align &A_2(a,b)\e\sum_{r=0}^5\binom{[\frac b
2]}r(-2)^re_{r+2}(a,b)\\
 &=e_2(a,b)-2[\frac b 2]e_3(a,b)+4\binom{[\frac b 2]}2e_4(a,b)-8\binom{[\frac b 2]}3e_5(a,b)+16\binom{[\frac b 2]}4e_6(a,b)
 \\&\q-32\binom{[\frac b 2]}2e_7(a,b)\\
 &\equiv
 -2a^3(b-1)(-17a^4+18a^2-1)+(b-1)(b-3)(-24a^5+48a^3-24a)\\&\q+4a^3(1-a^2)
 \mod{2^{10}}\endalign$$
and
$$\align A_1(a,b)&\e \sum_{r=0}^6\binom{[\frac b 2]}r(-2)^re_{r+1}(a,b)\\
&=e_1(a,b)-2[\frac b 2]e_2(a,b)+4\binom{[\frac b 2]}2e_3(a,b)-8\binom{[\frac b 2]}3e_4(a,b)+16\binom{[\frac b 2]}4e_5(a,b)
\\&\q-32\binom{[\frac b 2]}5e_6(a,b)+64\binom{[\frac b 2]}6e_7(a,b)\\
&\equiv
a-a^3-4a^3(b-1)(1-a^2)+a^3(b-1)(b-3)(-17a^4+18a^2-1)\\&\q+2(b-1)(b-3)(b-5)(b-7)(11a^5-22a^3+11a)\mod{2^{10}}.
 \endalign$$
 Thus,
 $$\align &\binom{2^{m-1}k-1}1\cdot 2\cdot A_2(a,b)\\
 &=2^mkA_2(a,b)-2A_2(a,b)\\
 &\equiv -2(4a^3(1-a^2)-2a^3(b-1)(-17a^4+18a^2-1)+(b-1)(b-3)
 (-24a^5+48a^3-24a))\\&\q+2^{m+2}ka(1-a^2)\mod{2^9}\endalign$$ and
 $$\align&\binom{2^{m-1}k-1}2\cdot (-1)^3\cdot \frac{2^3}3\cdot A_3(a,b)\\
 &=-\frac 4 3(2^{m-1}k-1)(2^{m-1}k-2)A_3(a,b)\equiv 2^{m+1}kA_3(a,b)-\frac8 3A_3(a,b)\\
 &\equiv -\frac 8 3(-34a^7+36a^5-2a^3)\mod{2^{11}}.
 \endalign$$
  Now combining the above congruences with Lemma 2.4, we deduce that
$$\align &E_{2^mk+b}^{(a)}-E_b^{(a)}\\&\equiv
2^{m-1}k\{-382a-44ab^4(a^4+1)+640a(a^2b^3-a^4-b)+704ab^3(a^4+1)
\\&\q+360ab^2-604a^3
-204a^7b-432a^5b-772a^3b+324a^5b^2+1330a^3b^2
\\&\q+34a^7b^2+88a^3b^4-422a^7+2^{m+2}ka(1-a^2)\}\mod{2^{m+10}}.\endalign$$
 To see the result, we note that $a^4\e b^4\e 1\mod{16}$,
$$5a(a^2b^3-a^4-b)\e 5a(a^2b^3-1-b)=5ab(a^2b^2-1)-5a\e
a^2b^2-1-5a\mod{16}$$ and
$$ab^3(a^4+1)=ab^3(a^4-1+2)\e a^4-1+2ab^3\e
a^4+2ab+2b^2-3\mod{32}.$$
 \end{proof}

\section{A congruence for $U_{k\varphi{(5^m)}+b} \mod{5^{m+5}}$}
\par For $n\in\Bbb N$ and $i\in\{0,1,\ldots,n\}$ let $s(n,i)$ be the Stirling number of the first kind given by
$x(x-1)\cdots (x-n+1)=\sum_{i=0}^n(-1)^{n-i}s(n,i)x^i.$
\begin{lem} $([2, p.197])$ Let $p$ be an odd prime, $n\in \Bbb N$ and let $f(k)$ be a p-regular
 function.
Let $A_m=p^{-m}\sum_{r=0}^m\binom mr(-1)^rf(r)$, $a_0=A_0$ and
$$a_i=(-1)^i\sum_{r=i}^{n-1}s(r,i)\frac{p^r}{r!}A_r\qtq{for}i=1,2,3,....$$
 Then
$$f(k)\equiv \sum_{i=0}^{n-1}a_ik^i\mod{p^n}.$$

\end{lem}

\begin{lem} Let $m\in\{5,6,7,\ldots\}$, $p\in\{3,5\}$ and
$s\in\{1,2,3,4,5,6\}$, and let $f(k)$ be a p-regular function. Then
$$f(p^{m-1}k)\e f(0)-p^{m-1}k\sum_{s=1}^6\frac1s\sum_{r=0}^s\binom sr
(-1)^{s-r}f(r)\mod{p^{m+5}}.$$
\end{lem}
\begin{proof} As $m\ge 5,$ we have $m\ge 7-s+ord_ps!$
 and so $m-1-ord_ps!\ge 6-s$ for $s\ge 2$. Thus, for $s\ge 2$ we have
 $\frac {p^{m-1}}{s!}\e
 0\mod{p^{6-s}}$ and so
$$\align  p^s\binom{p^{m-1}k}s&=p^{m-1+s}k\frac{(p^{m-1}k-1)(p^{m-1}k-2)\cdots(p^{m-1}k-s+1)}{s!}
\\&\e p^{m-1+s}k\frac{(-1)(-2)\cdots(-s+1)}{s!}=(-1)^{s-1}\frac{p^{m-1+s}k}s
\mod{p^{m+5}}.\endalign$$ Therefore, for $s=1,2,3,4,5,6$ we have
$$\align \binom{p^{m-1}k}s\sum_{r=0}^s\b sr(-1)^{s-r}f(r)&
\e -\frac{p^{m-1+s}k}s\cdot\f 1{p^s}\sum_{r=0}^s\b
sr(-1)^rf(r)\\&=-p^{m-1}k\frac1s\sum_{r=0}^s\b
sr(-1)^rf(r)\mod{p^{m+5}}.\endalign$$ Hence,
$$\sum_{s=1}^6\binom{p^{m-1}k}s\sum_{r=0}^s\b sr(-1)^{s-r}f(r)\e
-p^{m-1}k\sum_{s=1}^6\frac 1s\sum_{r=0}^s\b
sr(-1)^rf(r)\mod{p^{m+5}}.\tag 3.1$$
 As $p^{s-7}/s\in\Bbb
Z_p$ for $s\ge 7$, we see that
 $$\frac
1s\sum_{r=0}^s\binom s r(-1)^{s-r}f(s)\equiv 0\mod{p^7}\qtq{for}
s\ge 7.$$ Using the binomial inversion formula, we see that
$$\align f(p^{m-1}k)&=\sum_{s=0}^{p^{m-1}k}\binom {p^{m-1}k}s(-1)^s
\sum_{r=0}^s\binom sr(-1)^rf(r)\\
&=f(0)+p^{m-1}k\sum_{s=1}^{p^{m-1}k}\binom{p^{m-1}k-1}{s-1}
\f 1s\sum_{r=0}^s\binom sr(-1)^{s-r}f(r)\\
&\e f(0)+\sum_{s=1}^6\binom{p^{m-1}k}s\sum_{r=0}^s\binom
sr(-1)^{s-r}f(r)\mod{p^{m+6}}.\endalign$$ This together with (3.1)
yields the result.
\end{proof}

\begin{lem} Let $k$ be a nonnegative integer. Then
$$\align &(1+5^{4k})U_{4k}\equiv 6250k^6+50625k^5+51250k^4+59875k^3+72600k^2+7545k+2
\mod {5^7},\\&(1+5^{4k+2})U_{4k+2}\equiv
59375k^6+40625k^5+10625k^4+48875k^3+5575k^2+37500k-52 \mod
{5^7}.\endalign$$
\end{lem}
\begin{proof} Set $f(k)=(1+5^{4k+b)})U_{4k+b}$.
From [5, Theorem 2.1] we know that $f(k)$ is a $5-$regular function.
Thus  $A_m=5^{-m}\sum_{r=0}^m\binom mr(-1)^rf(r) \in \Bbb Z_5$ for
$m\in\{0,1,2,\ldots\}$.  It is easy to check that for $b=0$,
 $$\align &A_0=2,\q A_1\equiv 75371\mod{5^7},\q
A_2\e31378\mod{5^7},\q A_3\e73991\mod{5^7},\\&
A_4\e12133\mod{5^7},\q A_5\e36081\mod{5^7},\q
A_6\e43963\mod{5^7},\endalign$$ and that for $b=2,$
$$\align
&A_0=78073,\q A_1\equiv 6360\mod{5^7},\q A_2\e26626\mod{5^7},\q
A_3\e22469\mod{5^7},\\& A_4\e55958\mod{5^7},\q
A_5\e28490\mod{5^7},\q A_6\e28961\mod{5^7}.\endalign$$ Now applying
Lemma 3.1 we deduce the result.
\end{proof}

\begin{thm}Let $k,m\in \Bbb N,\ m\geq5$ and let $ b\in\{0,2,4,\ldots\}$. Then
$$\aligned &U_{k\varphi{(5^m)}+b}-(1+5^b)U_b\\&\equiv\cases
5^{m-1}k(7545+5050b-5375b^2+1250b^3+3125b^4+9375b^5) \pmod
{5^{m+5}}\\\qq\qq\qq\qq\qq\qq\qq\qq\qq\text{\rm{if} $b\equiv
0\mod4$},\\5^{m-1}k(-1575+5350b+2250b^2+7500b^3+3125b^5)
\mod{5^{m+5}}\\\qq\qq\qq\qq\qq\qq\qq\qq\qq\text{\rm{if} $b\equiv
2\mod4$}.
\endcases\endaligned
$$

\end{thm}

\begin{proof}

 From [5, Theorem 4.2] we know that $f(k)=(1-(\frac 5 3)5^{4k+b})U_{4k+b}=(1+5^{4k+b})U_{4k+b}$ is a 5-regular function.
  Let $r\in\{0,1,2,3,4,5,6\}.$ From Lemma 3.3 we see that for $b\equiv 0\mod 4$,
$$\align f(r)&=(1+5^{4r+b})U_{4r+b}\\&\equiv
6250\Big(r+\frac b4\Big)^6+50625\Big(r+\frac
b4\Big)^5+51250\Big(r+\frac b4\Big)^4+59875\Big(r+\frac b4\Big)^3
+72600\Big(r+\frac b4\Big)^2\\&\q+7545\Big(r+\frac b4\Big)+2 \mod
{5^7},
\endalign$$ and that for $b\equiv 2\mod 4$,
$$\align
 f(r)&=(1+5^{4r+b})U_{4r+b}
\\&\equiv59375\Big(r+\frac{b-2}4\Big)^6+40625\Big(r+\frac{b-2}4\Big)^5
+10625\Big(r+\frac{b-2}4\Big)^4+48875\Big(r+\frac{b-2}4\Big)^3
\\&\q+5575\Big(r+\frac{b-2}4\Big)^2+37500\Big(r+\frac{b-2}4\Big)-52
\mod{5^7}.\endalign$$
 Now combining the above with Lemma 3.2 gives the result.
\end{proof}

\section{A congruence for $E_{k\varphi{(5^m)}+b} \mod{5^{m+5}}$}
\begin{lem} Let $k$ be a nonnegative integer. Then
$$\align &(1-5^{4k})E_{4k}\equiv
31250k^6+11875k^5+18750k^4+64875k^3+54500k^2+50005k\mod{5^7},\\
&(1-5^{4k+2})E_{4k+2}\equiv
31250k^6+10625k^5+68750k^4+60375k^3+4625k^2+74290k+24\mod{5^7}.\endalign$$
\end{lem}
\begin{proof} From [3, Lemma 7.1] we know that $f(k)=(1-5^{4k+b})E_{4k+b}$
is a 5-regular function. Now one can prove the result by using Lemma
3.1 and doing some calculations.
\end{proof}
\begin{thm} Let $k,m\in\Bbb N$, $m\ge 5$ and $ b\in\{0,2,4,\ldots\}$.
Then
$$\aligned &E_{k\varphi{(5^m)}+b}-(1-5^b)E_b\\&\equiv\cases
5^{m-1}k\left(3130-4000b+3375b^2+3125b^3-3125b^4\right)
\mod{5^{m+5}}&\text{if $b\equiv
0\mod4$},\\5^{m-1}k(4790+1750b^2+3125b^3+6250b^4)\mod{5^{m+5}}&\text{if
$b\equiv 2\mod4$}.
\endcases\endaligned
$$
\end{thm}
\begin{proof}From [3, Lemma 7.1] we know that
$f(k)=(1-5^{4k+b})E_{4k+b}$ is a 5-regular function. By Lemma 3.2,
$$E_{k\varphi{(5^m)}+b}-(1-5^b)E_b
\e -5^{m-1}k\sum_{s=1}^6\f 1s\sum_{r=0}^s\binom s
r(-1)^{s-r}f(r)\mod{5^{m+5}}.\tag 4.1$$
\par Let $r\in\{0,1,2,3,4,5,6\}.$ From Lemma 4.1 we see that for
$b\equiv 0\mod 4$,
 $$\align
f(r)&\equiv 50005\Big(r+\frac b4\Big)+54500\Big(r+\frac
b4\Big)^2+64875\Big(r+\frac b4\Big)^3+18750\Big(r+\frac
b4\Big)^4\\&\q+11875\Big(r+\frac b4\Big)^5+31250\Big(r+\frac
b4\Big)^6\mod{5^7}.
 \endalign$$
and that for $b\equiv 2\mod 4$,
$$\align f(r)&\equiv
24+74290\Big(r+\frac {b-2}4\Big)+4625\Big(r+\frac
{b-2}4\Big)^2+60375\Big(r+\frac {b-2}4\Big)^3\\&\q+68750\Big(r+\frac
{b-2}4\Big)^4+10625\Big(r+\frac {b-2}4\Big)^5+31250\Big(r+\frac
{b-2}4\Big)^6\mod{5^7}.\endalign$$
 Now combining the above with (4.1) gives the result.
\end{proof}

\section{Congruences for $S_{k\varphi{(3^m)}+b}\mod{3^{m+5}}$}
\begin{lem} Let $k$ be a nonnegative integer. Then
$$\align &(1-3^{2k})S_{2k}\equiv 1620k^6+1620k^5+621k^4+981k^3+1179k^2+564k\mod{3^7},\\
&(1+3^{2k+1})S_{2k+1}\equiv
324k^6+2106k^5+1809k^4+1197k^3+549k^2+888k+2183\mod{3^7}.\endalign$$
\end{lem}
\begin{proof} From [4, Theorem 4.1] we know that $f(k)=\left(1-(-1)^b
3^{2k+b}\right)S_{2k+b}$ is a 3-regular function. Now one can prove
the result by using Lemma 3.1 and doing some calculations.
\end{proof}
\begin{thm}For $k,m\in\Bbb N$, $m\geq5$ and $ b\in\{0,1,2,\ldots\}$ we have
$$\aligned
&S_{k\varphi{(3^m)}+b}-(1-(-3)^b)S_b\\&\equiv\cases3^{m-2}k(-495
+1350b+567b^2-162b^3+972b^4-729b^5)\mod{3^{m+5}}\q&\text{\rm{if}
$2\mid
b$}\\3^{m-2}k(1422+1242b-891b^2-81b^3+243b^4+729b^5)\mod{3^{m+5}}
&\text{\rm{if} $2\nmid b$.}\endcases\endaligned$$
\end{thm}
 \begin{proof}From [4, Theorem 4.1(i)] we know that $f(k)=(1-(-1)^b
3^{2k+b})S_{2k+b}$ is a 3-regular function. By Lemma 3.2,
$$ S_{k\varphi{(3^m)}+b}-(1-(-3)^b)S_b\e-3^{m-1}k\sum_{s=1}^6\f
1s\sum_{r=0}^s\binom s r(-1)^{s-r}f(r)\mod{3^{m+5}}. \tag5.1$$
 Let $r\in\{0,1,2,3,4,5,6\}.$ From
Lemma 5.1 we see that for $b\equiv 0\mod2$,
$$\align f(r)&\equiv 1620\Big(r+\frac
b2\Big)^6+1620\Big(r+\frac b2\Big)^5+621\Big(r+\frac
b2\Big)^4+981\Big(r+\frac b2\Big)^3+1179\Big(r+\frac
b2\Big)^2\\&\q+564\Big(r+\frac b2\Big)\mod{3^7}\endalign$$ and that
for $ b\equiv 1\mod2$,
$$\align f(r)&\equiv 324\Big(r+\frac
{b-1}2\Big)^6+2106\Big(r+\frac {b-1}2\Big)^5+1809\Big(r+\frac
{b-1}2\Big)^4+1197\Big(r+\frac {b-1}2\Big)^3\\&\q+549\Big(r+\frac
{b-1}2\Big)^2+888\Big(r+\frac {b-1}2\Big)+2183\mod{3^7}\endalign$$
Combining the above with (5.1) yields the result.
\end{proof}

\section{Congruences for $S_{k\varphi{(5^m)}+b}\mod{5^{m+5}}$}
\begin{lem}Let $k$ and $b$ be nonnegative integers. Then
$$\aligned (1+5^{4k+b})S_{4k+b}\equiv\cases &33750k^5-6250k^4-14250k^3
+21500k^2\\&\q+930k+2\mod{5^7}\qq\qq\q\q\text{\rm{if} $b\equiv 0\mod4$},\\
&9375k^6+22500k^5+18750k^4+40000k^3+23525k^2\\&\q+7370k-6\mod{5^7}
\qq\qq\q\  \text{\rm{if} $b\equiv 1\mod4$
},\\&25000k^6+41875k^5+56250k^4+30875k^3+64650k^2\\&\q+44290k-78\mod{5^7}
\qq\qq\ \text{\rm{if} $b\equiv 2\mod4$
},\\&-3125k^6+40625k^5-9375k^4+67250k^3-8550k^2\\&\q+14525k+1386\mod{5^7}
\qq\q\ \text{\rm{if} $b\equiv 3\mod4$ }.\endcases\endaligned$$
\end{lem}
\begin{proof} From [4, Theorem 4.1(i)] we know that $f(k)=(1+5^{4k+b})S_{4k+b}$
is a 5-regular function. Now one can prove the result by using Lemma
3.1 and doing some calculations.
\end{proof}
\begin{thm} For $b,k,m\in\Bbb N$ and $m\geq 5$ we have
$$\aligned &S_{k\varphi{(5^m)}+b}-(1+5^b)S_b\\&\equiv\cases
5^{m-1}k(930-4875b-4625b^2-6250b^3-3125b^4)\mod{5^{m+5}} &\t{if
$4\mid
b$,}\\5^{m-1}k(4670+1450b+1250b^2+6250b^3+3125b^4+6250b^5)\mod
{5^{m+5}}&\t{if $4\mid
b-1$,}\\5^{m-1}k(-4235+3700b-3000b^2+6250b^3-9375b^4+6250b^5)\mod
{5^{m+5}}&\t{if $4\mid
b-2$,}\\5^{m-1}k(3725-1025b-2625b^2+6250b^3+3125b^5)\mod
{5^{m+5}}&\t{if $4\mid b-3$.}\endcases\endaligned$$
\end{thm}
\begin{proof} Set $f(k)=(1+5^{4k+b})S_{4k+b}$.  By Lemma 3.2,
$$S_{k\varphi({5^m})+b}-(1+5^b)S_b\equiv -5^{m-1}k\sum_{s=1}^6\f
1s\sum_{r=0}^s\binom s r(-1)^{s-r}f(r)\mod{5^{m+5}}.\tag6.1$$ Let
$r\in\{0,1,2,3,4,5,6\}.$ From Lemma 6.1 we see that for $b\equiv
0\mod 4$, $$\align f(r)&\equiv 33750\Big(r+\frac
b4\Big)^5-6250\Big(r+\frac b4\Big)^4-14250\Big(r+\frac
b4\Big)^3+21500\Big(r+\frac b4\Big)^2\\&\q+930\Big(r+\frac
b4\Big)+2\mod{5^7} ;\endalign$$ for $b\equiv 1\mod 4$,
$$\align f(r)&\equiv  9375\Big(r+\frac{b-1}4\Big)^6
+22500\Big(r+\frac{b-1}4\Big)^5
+18750\Big(r+\frac{b-1}4\Big)^4+40000\Big(r+\frac{b-1}4\Big)^3\\&\q+23525\Big(r+\frac{b-1}4\Big)^2
+7370\Big(r+\frac{b-1}4\Big)-6\mod{5^7};\endalign$$ for $b\equiv
2\mod 4$,
$$\align f(r)&\equiv
25000\Big(r+\frac{b-2}4\Big)^6+41875\Big(r+\frac{b-2}4\Big)^5
+56250\Big(r+\frac{b-2}4\Big)^4+30875\Big(r+\frac{b-2}4\Big)^3
\\&\q+64650\Big(r+\frac{b-2}4\Big)^2
+44290\Big(r+\frac{b-2}4\Big)-78\mod{5^7}; \endalign$$ and  for
$b\equiv 3\mod 4$,
$$\align f(r)&\equiv  -3125\Big(r+\frac{b-3}4\Big)^6
+40625\Big(r+\frac{b-3}4\Big)^5
-9375\Big(r+\frac{b-3}4\Big)^4+67250\Big(r+\frac{b-3}4\Big)^3
\\&\q-8550\Big(r+\frac{b-3}4\Big)^2
+14525\Big(r+\frac{b-3}4\Big)+1386\mod{5^7}.\endalign$$ Combining
the above with (6.1) we obtain the result.
\end{proof}


\begin{thebibliography}{99}
\bibitem{1} R. Ernvall, Generalized Bernoulli numbers, generalized
irregular primes, and class number, Ann. Univ. Turku. (Ser. A),
178(1979), 72 pp.
\bibitem{[S2]}Z.H.Sun, Congruences concerning Bernoulli numbers and Bernoulli
polynomials, Discrete Appl. Math. 105(2000),192-223.
\bibitem{[S4]}Z.H.Sun, Congruences involving Bernoulli polynomials, Discrete
 Math. 308(2008),71-112.
\bibitem{[S2012]}Z.H.Sun, Congruences for sequences similar to Euler numbers,
J. Number Theory 132(2012), 675-700.
\bibitem{[S3]}Z.H.Sun, Identities and congruences for a new sequence, Int. J. Number Theory 8(2012),207-225.
\bibitem{[S4]}Z.H.Sun, Some properties of a sequence analogous to Euler
numbers, Bull. Austral. Math. Soc. 87(2013), 425-440.
\bibitem{[S5]}Z.H.Sun, Lin-Lin Wang. An extension of Stern's congruence,
Int. J. Number Theory 9(2013), 413-419.
\bibitem{5}
Z.W. Sun, On Euler numbers modulo powers of two, J. Number Theory
115(2005),371-380.
\bibitem{2}
M.A. Stern, Zur Theorie der Eulerschen Zahlen, J. Reine Angew. Math.
 79(1875), 67-98.
\bibitem{6}
 S.S. Wagstaff Jr.,
Prime divisors of the Bernoulli and Euler numbers, in: M.A. Bennett
et al. (Eds.), Number Theory for the Millennium, vol. III (Urbana,
IL, 2000), A K Peters, Natick, MA, 2002, 357-374.
\end{thebibliography}
\end{document}